\documentclass[fleqn,reqno,11pt,a4paper,final]{amsart}

\usepackage[a4paper,left=30mm,right=30mm,top=30mm,bottom=30mm,marginpar=20mm]{geometry}
\usepackage{mathtools}     
\mathtoolsset{showonlyrefs}
\usepackage{xcolor} 
\usepackage{amsmath}
\usepackage{amssymb}
\usepackage{amsthm}
\usepackage{amscd}
\usepackage[ansinew]{inputenc}
\usepackage{cite}
\usepackage{bbm}
\usepackage{color}
\usepackage[english=american]{csquotes}
\usepackage[final]{graphicx}
\usepackage[colorlinks=true, pdfstartview=FitV, linkcolor=black, citecolor=black, urlcolor=black]{hyperref}
\usepackage{calc}
\usepackage{mathptmx}
\usepackage{t1enc}

\numberwithin{equation}{section}
\newtheoremstyle{thmlemcorr}{10pt}{10pt}{\itshape}{}{\bfseries}{.}{10pt}{{\thmname{#1}\thmnumber{ #2}\thmnote{ (#3)}}}
\newtheoremstyle{thmlemcorr*}{10pt}{10pt}{\itshape}{}{\bfseries}{.}\newline{{\thmname{#1}\thmnumber{ #2}\thmnote{ (#3)}}}
\newtheoremstyle{remexample}{10pt}{10pt}{}{}{\bfseries}{.}{10pt}{{\thmname{#1}\thmnumber{ #2}\thmnote{ (#3)}}}
\newtheoremstyle{ass}{10pt}{10pt}{}{}{\bfseries}{.}{10pt}{{\thmname{#1}\thmnumber{ A#2}\thmnote{ (#3)}}}

\theoremstyle{thmlemcorr}
\newtheorem{theorem}{Theorem}
\numberwithin{theorem}{section}
\newtheorem{lemma}[theorem]{Lemma}

\theoremstyle{remexample}
\newtheorem{remark}[theorem]{Remark}

\theoremstyle{ass}

\newcommand{\Dcal}{\mathcal{D}}

\newcommand{\T}{\mathbb{T}}

\DeclareMathOperator{\diverg}{div}

\DeclareMathOperator{\supp}{supp}

\newcommand{\norm}[1]{\|#1\|}

\newcommand{\R}{\mathbb{R}}

\newcommand{\Ocal}{\mathcal{O}}

\newcommand{\eps}{\epsilon}

\newcommand{\dx}{ dx}

\newcommand{\ueps}{u*\eta^\eps}
\newcommand{\pa}{\partial_\alpha}
\newcommand{\flux}{G_{i\alpha}}
\newcommand{\intO}{\int_{\Omega}}
\renewcommand{\eps}{\varepsilon}
\renewcommand{\epsilon}{\varepsilon}
\renewcommand{\phi}{\varphi}

\def\Xint#1{\mathchoice
{\XXint\displaystyle\textstyle{#1}}%
{\XXint\textstyle\scriptstyle{#1}}%
{\XXint\scriptstyle\scriptscriptstyle{#1}}%
{\XXint\scriptscriptstyle\scriptscriptstyle{#1}}%
\!\int}
\def\XXint#1#2#3{{\setbox0=\hbox{$#1{#2#3}{\int}$}
\vcenter{\hbox{$#2#3$}}\kern-.5\wd0}}

\def\dashint{\Xint-}

\begin{document}


\title[Entropy conservation]{On entropy conservation for general systems of conservation laws}

\author{Tomasz D\k{e}biec}
\address{\textit{Tomasz D\k{e}biec:}  Institute of Applied Mathematics and Mechanics, University of Warsaw, Banacha 2, 02-097 Warsaw, Poland}
\email{t.debiec@mimuw.edu.pl}

\begin{abstract}
In this work we consider companion conservation laws to general systems of conservation laws. We investigate sufficient regularity for weak solutions to satisfy companion laws, assuming the fluxes to be $C^{1,\gamma}$, $0<\gamma<1$, regular. We discuss applications of the general framework to nonlinear elasticity and compressible Euler system.
\end{abstract}

\vspace{6pt}

\maketitle




\section{Introduction}

Systems of differential equations describing the evolution of physical phenomena usually come with one or several conserved quantities, like entropy or energy. Such a formally conserved quantity may, however, fail to be conserved by a low--regular weak solution of the system. 
Recent years have seen intense research efforts investigating the relation between conservation of energy/entropy and regularity of weak solutions to a given physical system of equations. Notably, a general class of first--order conservation laws was considered in~\cite{GMSG}, and in~\cite{BaGwSGTiWi} on bounded domains.

A classical example of such a study is the Onsager's conjecture in the context of ideal fluids. It states that a weak solution of the three--dimensional incompressible Euler system will conserve energy if it is H\"older regular with exponent greater than $1\slash3$. Otherwise it is possible for anomalous dissipation of energy to occur. 
Both these assertions have been proved to be true: the former was first considered by Eyink~\cite{E94} and then fully proved by  Constantin et.\ al.\ ~\cite{CoETi94}, while the latter was recently shown by Isett~\cite{Is2018} and Buckmaster et.\ al.\ ~\cite{BuckmasterEtAl}. Investigating the possibility of analogous statements for other systems (particularly for compressible fluid dynamics) has become a lively direction of research.
Sufficient regularity conditions for the energy to be conserved were studied for a number of models: inhomogeneous incompressible Euler~\cite{ChenYu} and Navier--Stokes~\cite{LeSh2016}, compressible Euler~\cite{FGSW}, the full Euler system~\cite{DrEy2018}, compressible Navier--Stokes~\cite{Yu2017}, or Euler--Korteweg~\cite{DGST}.

In this paper we consider the general system of $n$ conservation laws
\begin{equation}\label{eq:ConsSystem}
    \partial_{\alpha} G_{i\alpha}(u(x)) = 0 
\end{equation}
 where $u:\Omega\to\Ocal$ is an unknown quantity and $G_{i\alpha}:\Ocal\to\R$,\; $i=1,\dots,n$,\; $\alpha=0,\dots,d$ are the fluxes. We assume that $\Omega\subset\R^{d+1}$ is open and $\Ocal\subset\R^n$ is open and convex. In the above $x=(x_0,\dots,x_d)$, $\partial_\alpha=\partial_{x_\alpha}$, and we use the summation convention. 
We are concerned with the validity of the following \emph{companion law}
\begin{equation}\label{eq:CL}
    \partial_\alpha Q_\alpha(u(x)) = 0,
\end{equation}
where $Q_\alpha:\Ocal\to\R$,\; $\alpha=0,\dots,d$.
We require that $Q_\alpha$ be related to $G_{i\alpha}$ via the existence of smooth functions $B_i:\Ocal\to\R$,\; $i=1,\dots,n$ such that the following compatibility conditions are met
\begin{equation}\label{eq:Compatibility}
    \partial_jQ_\alpha(u) = B_i(u)\partial_jG_{i\alpha}(u),\;\;\; \text{for $j=1,\dots,n$, $\alpha=0,\dots,d$}.
\end{equation}
Further, we assume throughout that for some $\gamma\in(0,1)$ the fluxes $Q_\alpha$ of~\eqref{eq:CL} satisfy the growth condition
\begin{equation}\label{eq:growth}
    |Q_\alpha(u)| \leq C(1+|u|^{2+\gamma}),\quad\text{for each}\; u\in\Ocal.
\end{equation}

\noindent In the fluid mechanics applications mentioned above the \emph{companion law}~\eqref{eq:CL} represents the energy equality.
It is well known that~\eqref{eq:CL} holds for any classical solution $u$ of~\eqref{eq:ConsSystem}, i.e.\ for any Lipschitz vector field $u$ which satisfies~\eqref{eq:ConsSystem} for a.e.\ $x\in\Omega$. 
Indeed, this follows from multiplying~\eqref{eq:ConsSystem} by $B_i(u)$, summing over $i=1,\dots,n$, and invoking the chain rule.
However, if $u$ is merely an irregular \emph{weak solution}, i.e.\ satisfying
\begin{equation}
    \intO G_{i\alpha}(u(x))\partial_\alpha\phi(x)\ dx = 0,\;\;\;\text{for all}\;\phi\in C_c^1(\Omega),\;\;i=1,\dots,n,
\end{equation}
then the additional conservation law~\eqref{eq:CL} can break down, even in the scalar case. For hyperbolic systems this is typically due to shock waves, giving rise to discontinuous solutions.
It is found in~\cite{GMSG} that if $B_i\in W^{1,\infty}$ and $G_{i\alpha}$ have bounded second derivatives, then any solution $u\in B_{3,\infty}^{s}$ to~\eqref{eq:ConsSystem} will satisfy~\eqref{eq:CL}, provided $s>1\slash3$.
Importantly, this abstract framework covers particular physical first--order systems (Euler equations, magnetohydrodynamics, polyconvex elasticity).
In this paper we focus on providing a sufficient regularity condition so that a weak solution conserves the companion quantity, assuming the fluxes to be no more regular than $C^{1,\gamma}$ for some $0<\gamma<1$.
Of course, at this level of generality one cannot hope for the assumptions to be universally optimal, but with the extra structure of a problem under consideration one can improve the statement (e.g.\ require Besov regularity only with respect to some variables). Furthermore, in~\cite{BaGwSGTiWi2} a similar extension of Onsager's conjecture to general conservation laws is done in presence of physical boundaries, introducing a more optimal Besov-$VMO$ type space, see also~\cite{BaGwSGTiWi, BaTi, BaTiWi}.

\medskip

In the known works on the subject~\cite{FGSW, GMSG, BaGwSGTiWi} it is a crucial assumption that the nonlinearities be of class $C^2$ in the range of the dependent variables. This allows to treat them as quadratic expressions in the relevant commutator estimates.
Our objective in this work is to investigate sharp regularity assumptions which would allow to guarantee conservation of companion quantities for a general first--order system of conservation laws whose flux is H\"older--continuously differentiable. One important example to keep in mind is nonlinear elasticity, where the gradient of the stored energy functional is often not a $C^2$ function.
Further, one can think of the compressible Euler system with vacuum as done recently in~\cite{ADSW}.

One of the major differences between incompressible and compressible fluid dynamics is the possible formation of \emph{vacuum} in the latter case, i.e.\ density of the fluid becoming zero in some region. Consider the isentropic compressible Euler system  
\begin{equation}\label{eq:compressibleEuler}
\begin{aligned}
\partial_t(\rho v)+\diverg(\rho v\otimes v)+\nabla p(\rho)&=0,\\
\partial_t\rho+\diverg(\rho v)&=0.
\end{aligned}
\end{equation}
It is classically known that conservation laws like~\eqref{eq:compressibleEuler} may develop singularities (shocks) in finite time, which prohibits the use of a smooth notion of solution. Rather, one works with solutions in the sense of distributions, which may be very rough.
For a smooth solution, with density initially bounded away from zero, it would easily follow from the continuity equation (see e.g.\ \cite{DiLi1989}) that $\rho$ remains bounded away from zero for all times.
More precisely, this requires $v$ to have bounded divergence. However, there seems to be no way to guarantee that the velocity component of a weak solution of~\eqref{eq:compressibleEuler} has bounded divergence, and thus it can not be excluded that the solution spontaneously develops vacuum in finite time. In fact, to our knowledge it remains an outstanding open question whether this can actually occur for the compressible Euler or even Navier--Stokes equations.  

The formation of vacuum constitutes a degeneracy that, in many situations, vastly complicates the mathematical analysis of compressible models. For the system~\eqref{eq:compressibleEuler}, a typical and physically reasonable pressure law is the polytropic pressure $p(\rho)=\rho^{\gamma_0}$ with $\gamma_0>1$. However, in the most relevant regime $1<\gamma_0<2$ the second derivative, being of order $\rho^{\gamma_0-2}$, blows up at zero.
In~\cite{ADSW} a number of sufficient conditions is given for weak solutions of~\eqref{eq:compressibleEuler} to ensure energy conservation even after vacuum formation. To the best of our knowledge, the only other result on energy conservation for non--$C^2$ nonlinearities is the one on active scalar equations~\cite{AkWi2019}, using however different techniques.

In the current paper we study relaxation of the $C^2$ assumption on the nonlinearities in the context of general conservation laws and their companion quantities. We prove an analogue of Theorem~1.1 in~\cite{GMSG}, using the function space framework of~\cite{BaGwSGTiWi2}. Our main result is the following.

\begin{theorem}\label{thm:main1}
Let $0<\gamma<1$ and let $Q_\alpha$ satisfy~\eqref{eq:growth}. Suppose $u$ is a weak solution to~\eqref{eq:ConsSystem}. Assume that the functions $B_i$ in~\eqref{eq:Compatibility} belong to $W^{1,\infty}(\Ocal;\R)$. Assume further that $G_{i\alpha}\in C^{1,\gamma}(\Ocal,\R)$ and $u\in \underline{B}_{(\gamma+2),VMO}^{1\slash {(\gamma+2)}}(\Omega;\Ocal)$.
Then~\eqref{eq:CL} is satisfied in the sense of distributions.
\end{theorem}

In Section~\ref{sec:prelim} we define the relevant function spaces, and then in Section~\ref{sec:result} give a proof of the above theorem.
In Section~\ref{sec:examples} we look at two examples for possible application of Theorem~\ref{thm:main1}: 
in Subsection~\ref{ssec:elasticity} we apply our main result to the equations of nonlinear elasticity and in Subsection~\ref{ssec:Euler} we discuss the scope and limitations of the current study in the context of isentropic compressible Euler equations. 

\section{Preliminaries and notation}

\label{sec:prelim}
Let $p\geq1$ and $s\in(0,1)$.
We denote by $B_{p,\infty}^s(\Omega;\Ocal)$ the space of those functions $u\in L^p$, whose translations can be controlled as follows
\begin{equation}
    \sup\limits_{x'\in\Omega}|x'|^{-s}\norm{u(\cdot)-u(\cdot-x')}_p < \infty,
\end{equation}
so that $B_{p,\infty}^s(\Omega;\Ocal)$ can be equipped with the norm
\begin{equation}
    \norm{u}_{B_{p,\infty}^s(\Omega;\Ocal)}\coloneqq \norm{u}_p + \sup\limits_{x'\in\Omega}|x'|^{-s}\norm{u(\cdot)-u(\cdot-x')}_p.
\end{equation}
One then clearly has
\begin{equation}\label{eq:besovshift}
\norm{u(\cdot)-u(\cdot-x')}_p \leq |x'|^s\ \norm{u}_{B_{p,\infty}^s(\Omega;\Ocal)}.
\end{equation}
Let $\eta\in C_c^\infty(\R^{d+1})$ be a non--negative symmetric function of unit mass with $\supp\eta\subset \overline{B_1(0)}$ and set $\eta^\eps(x) = \eps^{-(d+1)}\eta(\frac{x}{\eps})$. Then one easily shows the following estimates, cf.~\cite[Lemma~2.1]{DGST}
\begin{equation}\label{eq:besovapprox}
\norm{\ueps-u}_p \leq C\eps^s\norm{u}_{B_{p,\infty}^s(\Omega;\Ocal)}
\end{equation}
and
\begin{equation}\label{besovgradient}
\norm{\nabla(\ueps)}_p \leq C\eps^{s-1}\norm{u}_{B_{p,\infty}^s(\Omega;\Ocal)}.
\end{equation}

\medskip

The Besov space functional setting is usual in the context of sufficient conditions for energy conservation. Already in~\cite{CoETi94} the Onsager conjecture is shown in the Besov space $B_{3,\infty}^s$ rather than $C^s$, $s>1\slash3$. Slightly weaker assumptions have then been found in~\cite{Chetal2008}, where the space $B_{3,c_0}^{1\slash3}(\R^3)$ is considered. We say that $u$ belongs to $B_{p,c_o}^s(\Omega;\Ocal)$ if
\begin{equation}
    \lim\limits_{|y|\to0}\;\int_{\Omega}\frac{|u(x+y)-u(x)|^p}{|y|^{ps}}\dx = 0.
\end{equation}

Fjordholm and Wiedemann~\cite{FjWi} use instead the even weaker integral condition
\begin{equation}\label{FjWicondition}
    \liminf\limits_{\eps\to0}\frac1\eps\int_0^T\int_{\T^d}\dashint_{B_{\eps}(x)}|u(x,t)-u(y,t)|^3\ dydxdt = 0.
\end{equation}
to prove local energy conservation for the incompressible Euler, see also~\cite{DuRo2000, Shvydkoy2009}.
Inspired by~\eqref{FjWicondition}, Bardos et al.\ \cite{BaGwSGTiWi2} introduce the Besov-$VMO$ type space $\underline{B}_{3,VMO}^{1\slash3}$. We use here this construction with suitable modifications. We will say that $u$ belongs to $\underline{B}_{p,VMO}^{1\slash p}(\Omega;\Ocal)$ if
\begin{equation}\label{Besov-VMOcondition}
    \frac1\eps\int_{\Omega'}\dashint_{B_\eps(x)}|u(x)-u(y)|^p\ dy\dx \leq \omega_{\Omega'}(\eps),
\end{equation}
where $\Omega'\subset\subset\Omega$ and $\omega_{\Omega'}$ is positive and satisfies $\liminf\limits_{\eps\to0}\omega_{\Omega'}(\eps)=0$.
For $s>1\slash p$ we then have the chain of inclusions
\begin{equation}
    C^s \subset B_{p,\infty}^s \subset B_{p,c_0}^{1\slash p} \subset \underline{B}_{p,VMO}^{1\slash p} \subset B_{p,\infty}^{1\slash p}.
\end{equation}
Thus our results, proven for solutions in the Besov-$VMO$ spaces will {\em a fortiori} work also for solutions in classical H\"older and Besov spaces.

Similarly to property~\eqref{besovgradient} we have
\begin{lemma}\label{Lemma:vmogradient}
Let $1<p<\infty$ and $u\in\underline{B}_{p,VMO}^{1\slash p}(\Omega;\Ocal)$. Then
\begin{equation}
    \norm{\partial_\alpha(u*\eta^\eps)}_{L^p(\Omega')} \leq C(\omega_{\Omega'}(\eps))^{1\slash p}\eps^{-1\slash p'}
\end{equation}
where $\Omega'\subset\subset\Omega$ and $1\slash p + 1\slash p' = 1$.
\end{lemma}
\begin{proof}
We use Jensen's inequality and properties of the standard mollifier to get
\begin{equation}
    \begin{split}
        \norm{\partial_\alpha(u*\eta^\eps)}_{L^3(\Omega')}^p &=\int_{\Omega'}\left|\int_{\R^{d+1}}u(y)\frac{\partial}{\partial x_\alpha}\eta^\eps(x-y)\ dy \right|^p\dx\\
        &=\int_{\Omega'}\left|\int_{\R^{d+1}}(u(x-y)-u(x))\frac{\partial}{\partial y_\alpha}\eta^\eps(y)\ dy \right|^p\dx\\
        &\leq C(\eta)\eps^{1-p}\int_{\Omega'}\int_{\R^{d+1}}|u(x-y)-u(x)|^p\frac{\partial}{\partial y_\alpha}\eta^\eps(y)\ dy\dx \\
        &\leq C\eps^{1-p}\eps^{-1}\int_{\Omega'}\dashint_{B_{\eps}(0)}|u(x-y)-u(x)|^p\ dy\dx\\
        &\leq C\eps^{1-p}\omega_{\Omega'}(\eps).
    \end{split}
\end{equation}

\end{proof}

\section{Proof of the main result}\label{sec:result}
In this section we prove the main result of this article, Theorem~\ref{thm:main1}. The proof relies upon mollification of equation~\eqref{eq:ConsSystem}, and then estimation of the resulting commutators $\flux(\ueps)-\flux(u)*\eta^\eps$.
The observation of Feireisl et al.\ in~\cite{FGSW} was that if a nonlinear function $G$ is twice continuously differentiable with respect to each dependent variable, then, by means of a Taylor expansion, one can treat the commutator $G(\ueps)-G(u)*\eta^\eps$ as a bilinear term, obtaining a bound $|G(\ueps)-G(u)*\eta^\eps|\lesssim |u^\eps-u|^2$. It was then observed by Akramov et al. in \cite{ADSW} that a similar estimate holds for $G$ in $C^{1,\gamma}$. We follow this approach here with the following lemma.

\begin{lemma}\label{Lemma:Commutator}
Let $1\leq q<\infty$. Suppose $G:\Ocal\to\R$ is continuously differentiable with $\gamma$-H\"older continuous partial derivatives, and let $u\in \underline{B}_{q(\gamma+1),VMO}^{1\slash q(\gamma+1)}(\Omega;\Ocal)$. Then
\begin{equation}
    \norm{G(\ueps)-G(u)*\eta^\eps}_{L^q(\Omega')} \leq C(\eps \omega_{\Omega'}(\eps))^{1\slash q},
\end{equation}
where $\Omega'\subset\subset\Omega$ and $\eps>0$ is small enough.
\end{lemma}
\begin{proof}
For any $y,y_0\in\Ocal$ we parameterize the line segment $[y_0,y]$ by $l(t)=y_0+t(y-y_0)$. We then have
\begin{equation}
\begin{aligned}
    G(y) - G(y_0) &= \int_{[y_0,y]}\nabla G(z)\cdot dz = \int_0^1\nabla G(l(t))\cdot(y-y_0)\ dt\\
    &=\int_0^1\left(\nabla G(l(t)) - \nabla G(y_0)\right)\cdot(y-y_0)\ dt + \nabla G(y_0)\cdot(y-y_0).
\end{aligned}
\end{equation}
Therefore
\begin{equation}
\begin{aligned}
    |G(y)-G(y_0)-\nabla G(y_0)\cdot(y-y_0)|&\leq \int_0^1|\left(\nabla G(l(t)) - \nabla G(y_0)\right)\cdot(y-y_0)|\ dt\\
    &\leq C\int_0^1|l(t)-y_0|^\gamma|y-y_0|\ dt\\
    &\leq C|y-y_0|^{\gamma+1}
\end{aligned}
\end{equation}
with the constant $C$ independent of the choice of $y_0,y$. We deduce that 
\begin{equation}\label{eq:T1}
    |G(\ueps)-G(u)-\nabla G(u)\cdot(\ueps-u)|\leq C|\ueps-u|^{\gamma+1}
\end{equation}
and
\begin{equation}
    |G(u(x'))-G(u(x))-\nabla G(u(x))\cdot(u(x')-u(x))|\leq C|u(x')-u(x)|^{\gamma+1}.
\end{equation}
Multiplying the last inequality by $\eta^\eps(x-x')$, integrating in $x'\in\Ocal$ and applying Jensen's inequality to the left-hand side yields
\begin{equation}\label{eq:T2}
    |G(u)*\eta^\eps-G(u)-\nabla G(u)\cdot(\ueps-u)|\leq C\intO |u(x)-u(x-x')|^{\gamma+1}\eta^\eps(x')\ dx'
\end{equation}
Combining inequalities~\eqref{eq:T1} and~\eqref{eq:T2} and using Jensen's inequality we get
\begin{equation}\label{eq:FluxCommutator}
    |G(u)*\eta^\eps-G(\ueps)|\leq C\intO |u(x)-u(x-x')|^{\gamma+1}\eta^\eps(x')\ dx'.
\end{equation}
Therefore
\begin{equation}
\begin{split}
    \norm{G(\ueps)-G(u)*\eta^\eps}_{L^q(\Omega')}^q &\leq C\int_{\Omega'}\left|\intO |u(x)-u(x-x')|^{\gamma+1}\eta^\eps(x')\ dx' \right|^q \dx\\
    &\leq C\int_{\Omega'}\intO |u(x)-u(x-x')|^{(\gamma+1)q}\eta^\eps(x')\ dx' \dx\\
    &\leq C\int_{\Omega'}\dashint_{B_\eps(x)} |u(x)-u(x')|^{(\gamma+1)q}\ dx' \dx\\
    &\leq C\eps\omega_{\Omega'}(\eps).
\end{split}
\end{equation}
\end{proof}

\noindent Armed with Lemma~\ref{Lemma:Commutator} we prove the main theorem.

\begin{proof}[Proof of Theorem~\ref{thm:main1}]
Taking the convolution of equation~\eqref{eq:ConsSystem} with the mollifier $\eta^\eps$ as above we get
\begin{equation}\label{eq:smoothentropy}
    \pa\flux(\ueps) = \pa\left(\flux(\ueps) - \flux(u)*\eta^\eps\right).
\end{equation}
Multiplying with $B_i(\ueps)$ and summing over indicies $i=1,\dots,n$ the last equality becomes
\begin{equation}
    \pa Q_\alpha(\ueps) = B_i(\ueps)\pa\left(\flux(\ueps) - \flux(u)*\eta^\eps\right)
\end{equation}
where the left hand side comes from the compatibility condition~\eqref{eq:Compatibility}.

We now wish to pass to the limit, in the sense of distributions, as $\eps\to 0$.
The growth condition on $Q_\alpha$ implies that the family $(Q_\alpha(\ueps))_{\eps>0}$ is uniformly integrable and tight. Therefore by Vitali's convergence theorem, we have the convergence of integrals
\begin{equation}
    \intO\pa Q_\alpha(\ueps)\phi\ dx\to\intO\pa Q_\alpha(u)\phi\ dx
\end{equation}
for any test function $\phi\in\Dcal(\Omega)$. Therefore left-hand side of~\eqref{eq:smoothentropy} converges in $\Dcal'(\Omega)$ to $\pa Q_\alpha(u)$.

We will now discuss the convergence
\begin{equation}
     B_i(\ueps)\pa\left(\flux(\ueps) - \flux(u)*\eta^\eps\right)\to 0
\end{equation}
in $\Dcal'(\Omega)$. This is the main technical step of the proof. 

\medskip

Choose a test function $\phi\in\Dcal(\Omega)$ supported in $\Omega'\subset\subset\Omega$ and consider
\begin{equation}
    R_\eps = \intO\phi(x)B_i(\ueps)\pa\left(\flux(\ueps) - \flux(u)*\eta^\eps\right)\ \dx.
\end{equation}
Integration by parts yields
\begin{equation}
    \begin{aligned}
    R_\eps = -\intO\pa\phi B_i(\ueps)&\left(\flux(\ueps) - \flux(u)*\eta^\eps\right)\ \dx\\
    &- \intO\phi\pa(B_i(\ueps))\left(\flux(\ueps) - \flux(u)*\eta^\eps\right)\ \dx
    \end{aligned}
\end{equation}
and thus
\begin{equation}
    \begin{aligned}
    |R_\eps| \leq \norm{\pa\phi}_\infty\int_{\Omega'} &|B_i(\ueps)||\flux(\ueps) - \flux(u)*\eta^\eps|\ \dx\\
    &+ \norm{\phi}_\infty\int_{\Omega'}|\nabla B_i(\ueps)\cdot\pa(\ueps)||\flux(\ueps) - \flux(u)*\eta^\eps|\ \dx.
    \end{aligned}
\end{equation}
We conclude the estimation of the commutator $R_\eps$ using Lemmata~\ref{Lemma:vmogradient} and~\ref{Lemma:Commutator}.
Firstly, we have
\begin{equation}
\intO |B_i(\ueps)||\flux(\ueps) - \flux(u)*\eta^\eps|\ \dx \leq C\norm{B}_{L^\infty(\Omega)}(\omega_{\Omega'}(\eps)\eps)^{\frac{\gamma+1}{\gamma+2}},
\end{equation}
which converges to zero as $\eps\to 0$.
Secondly,
\begin{equation}
\begin{aligned}
\intO|\nabla& B_i(\ueps)\cdot\pa(\ueps)||\flux(\ueps) - \flux(u)*\eta^\eps|\ \dx\\
&\leq\norm{\nabla B_i}_{L^\infty(\Omega)}\norm{\pa(\ueps)}_{L^{\gamma+2}(\Omega')}\norm{\flux(\ueps) - \flux(u)*\eta^\eps}_{L^{\frac{\gamma+2}{\gamma+1}}(\Omega')}\\
&\leq C\norm{B}_{W^{1,\infty}}(\omega_{\Omega'}(\eps))^{\frac{1}{\gamma+2}}\eps^{-\frac{\gamma+1}{\gamma+2}}(\omega_{\Omega'}(\eps)\eps)^{\frac{\gamma+1}{\gamma+2}}\\
&\leq C\omega_{\Omega'}(\eps)
\end{aligned}
\end{equation}
which converges to zero as $\eps\to 0$.
\end{proof}

\begin{remark}
\begin{enumerate}
    \item As mentioned before, Theorem~\ref{thm:main1} implies {\em a fortiori} the desired conclusion for solutions $u$ in the space $B_{\gamma+2,\infty}^s(\Omega;\Ocal)$ with $s>1\slash{(\gamma+2)}$.
    \item The theorem provides a sufficient condition for the conservation of a companion quantity. Since shock solutions belong to $B_{p,\infty}^{1\slash p}$ for any $1\leq p <\infty$ and we have $\underline{B}_{p,VMO}^{1\slash p}\subset B_{p,\infty}^{1\slash p}$, there is little room for improvement in terms of optimality of this condition.
    \item Given additional information on the problem~\eqref{eq:ConsSystem}, the assumption on the solution $u$ may be partially relaxed, cf.\ \cite[Remark~2.4]{BaGwSGTiWi2}.
\end{enumerate}
\end{remark}

\section{Examples}\label{sec:examples}

\subsection{Nonlinear elasticity}\label{ssec:elasticity}

Consider a continuous medium in three dimensions with a nonlinear elastic response, described by
\begin{equation}\label{eq:nonlinearelasticity}
\partial^2_{t}y_i = \partial_\alpha T_{i\alpha}(\nabla y)
\end{equation}
where $y:(0,T)\times\T^3\to\R^3$ is a motion and $T = (T_{i\alpha})$, $i,\alpha = 1,2,3$, is the Piola--Kirchoff stress tensor. We shall assume that the material is hyperelastic, meaning that $T$ arises as the gradient of a stored energy function $W:\R^{3\times 3}\to[0,\infty)$
\[
T_{i\alpha}(F) = \frac{\partial W(F)}{\partial F_{i\alpha}}.
\]
Writing $v_i = \partial_ty_i$ for velocity components and $F_{i\alpha} = \partial_\alpha y_i$ for the deformation gradient, the system~\eqref{eq:nonlinearelasticity} can be written as a system of $12$ conservation laws
\begin{equation}\label{eq:elasticconservlaw}
    \begin{aligned}
    \partial_tv_i &= \partial_\alpha T_{i\alpha}(F) = \partial_\alpha\left(\frac{\partial W(F)}{\partial F_{i\alpha}} \right),\;\;\;i=1,2,3,\\
    \partial_tF_{i\alpha} &= \partial_\alpha v_i = \partial_{\beta}(v_i\delta_{\alpha\beta}),\;\;\;i,\alpha=1,2,3.
    \end{aligned}
\end{equation}
One requires of solutions to~\eqref{eq:elasticconservlaw} that $F=\nabla y$, so that~\eqref{eq:elasticconservlaw} is equivalent to~\eqref{eq:nonlinearelasticity}. Equivalently
\[
\partial_\beta F_{i\alpha} = \partial_\alpha F_{i\beta},
\]
and in fact it is enough to require this condition only at time $t=0$, since it is then transported by the equation.
For the system~\eqref{eq:elasticconservlaw} we have the following companion law, representing conservation of total mechanical energy,
\begin{equation}\label{eq:entropyconservaton}
    \partial_t\left(\frac{1}{2}|v|^2 + W(F)\right) + \partial_\alpha\left(-v_i\frac{\partial W(F)}{\partial F_{i\alpha}}\right) = 0.
\end{equation}
To obtain this relation we multiply the equations by $v_i$ and $\frac{\partial W(F)}{\partial F_{i\alpha}}$, respectively, and sum over corresponding indices.
Notice that the fluxes in~\eqref{eq:elasticconservlaw} are smooth w.r.t.\ the dependent variables $v$ and $F$, apart from the term involving $\frac{\partial W(F)}{\partial F_{i\alpha}}$, whose regularity depends on the regularity of the stored energy $W$. To apply Theorem~\ref{thm:main1}, we require that $W$ be of class $C^{2,\gamma}$, so that $\frac{\partial W(F)}{\partial F_{i\alpha}}\in C^{1,\gamma}$. Importantly, then the multiplier $\frac{\partial W(F)}{\partial F_{i\alpha}}$ is in $W^{1,\infty}$, so all the requirements of the theorem are met. We thus have
\begin{theorem}\label{thm:elasticity}
Let $0<\gamma<1$. Suppose $W\in C^{2,\gamma}(\R^{3\times3})$. Let $(v,F)$ be a distributional solution to~\eqref{eq:elasticconservlaw} with
\begin{equation}
    v_{i}, F_{i\alpha} \in \underline{B}_{\gamma+2,VMO}^{1\slash(\gamma+2)}((0,T)\times\T^3).
\end{equation}
Then the companion law~\eqref{eq:entropyconservaton} is satisfied in the sense of distributions on $(0,T)\times\T^3$.
\end{theorem}

\begin{remark}
In the above discussion we ignore the question of convexity of the domain of $W$. In fact, to guarantee that the deformation is orientation--preserving one requires that $\det F>0$, so that $W$ should really be defined only on the non-convex set $\R^{3\times3}_+$ rather than $\R^{3\times 3}$, see~\cite{Ball1977} for further discussion and references. On the other hand, in the proof of Lemma~\ref{Lemma:Commutator} we rely on the convexity of the set $\Ocal$. However, the issue of non-convexity of the domain of the flux has been discussed in~\cite[Theorem~1.2]{GMSG}, and the same argument carries over to the current framework. We therefore skip the details here.     
\end{remark}

\subsection{Compressible Euler}\label{ssec:Euler}
We now wish to discuss the extent of possible application of our main result to a compressible inviscid fluid. To this end let us consider the isentropic Euler system
\begin{equation}\label{eq:compressEuler}
    \begin{aligned}
    \partial_t\rho + \partial_j(\rho v_j) &= 0,\\
    \partial_t(\rho v_i) + \partial_j(\rho v_iv_j + p(\rho)\delta_{ij}) &= 0,\;\;\; i=1,\dots, d
    \end{aligned}
\end{equation}
in $(0,T)\times\T^d$, where $\rho:(0,T)\times\T^d\to\R^d$ is the density and $v=(v_1,\dots,v_d):(0,T)\times\T^d\to\R^d$ is the velocity of the fluid. The pressure $p=p(\rho)$ is a given function of the density.

\smallskip

\noindent In~\cite{FGSW} it is shown that if $v$ belongs to $B^{s_1}_{3,\infty}((0,T)\times\T^d)$, while $\rho$ and $\rho v$ are in $B^{s_2}_{3,\infty}((0,T)\times\T^d)$ with $\max{(2s_1+s_2, s_1+2s_2)}>1$, and if $p\in C^2([\underline{\rho},\overline{\rho}])$, where $0<\underline{\rho}\leq\rho\leq\overline{\rho}$, then the energy is locally conserved in $(0,T)\times\T^d$.

\smallskip

\noindent In the context of~\eqref{eq:ConsSystem} let $\Omega = (0,T)\times\T^d$, $\Ocal = \R_+\times\R^d$, $x = (x_0=t, x_1, x_2,\dots,x_d)$ and $u(x) = (u^1, u^2) = (\rho, v)$.
Then, we can rewrite~\eqref{eq:compressEuler} as
\begin{equation}
    \partial_jG_{ij}(u(x)) = 0,\;\;\; i=0,1,\dots,d
\end{equation}
with
\begin{equation}
     G_{i0} = 
  \begin{cases} 
  u^1  & i=0 \\
  u^1u^2_i & i\geq 1,
  \end{cases}
  \;\;\;\;\;\;\text{and}\;\;\;\;\;\;
  G_{ij} = 
  \begin{cases} 
  u^1u^2_j  & i=0,\,j\geq 1 \\
  u^1u^2_iu^2_j + p(u^1)\delta_{ij} & i,j\geq 1.
  \end{cases}
\end{equation}
The companion law~\eqref{eq:CL}, representing total energy conservation, takes the form
\begin{equation}\label{eq:eulerenergy}
    \partial_t\left(\frac12\rho |v|^2 + P(\rho)\right) + \partial_j\left((\frac12\rho |v|^2 + P(\rho) + p(\rho))u_j\right) = 0,
\end{equation}
i.e.
\begin{equation}
    \partial_jQ_j(u(x)) = 0
\end{equation}
with
\begin{equation}
    Q_j(u) = 
     \begin{cases} 
  \frac12u^1 |u^2|^2 + P(u^1), & i=0, \\
  (\frac12u^1 |u^2|^2 + P(u^1) + p(u^1))u^2_j, & j\geq 1.
  \end{cases}
\end{equation}
Here $P(\rho) = \rho\int_1^\rho\frac{p(r)}{r^2}\,dr$ is the pressure potential.
The compatibility conditions~\eqref{eq:Compatibility} are satisfied with
\begin{equation}
    B_i(u) = 
         \begin{cases} 
  -\frac12|u^2|^2 + P'(u^1), & i=0, \\
  u^2_i, & i\geq 1.
  \end{cases}
\end{equation}
Notice that when $\rho\geq\underline{\rho}>0$ and $p\in C^{1,\gamma}$, then $P'\in W^{1,\infty}$. Therefore we are in position to apply Theorem~\ref{thm:main1} and we can state
\begin{theorem}\label{thm:euler1}
Let $0<\gamma<1$. Let $(\rho, u)$ be a distributional solution to~\eqref{eq:compressEuler} such that
\begin{equation}
    \begin{aligned}
    \rho, u &\in \underline{B}_{\gamma+2,VMO}^{1\slash(\gamma+2)}((0,T)\times\T^d),\;\;\;
    0<\underline{\rho}\leq\rho\leq\overline{\rho}<\infty.
    \end{aligned}
\end{equation}
Assume that $p\in C^{1,\gamma}([\underline{\rho},\overline{\rho}])$.
Then the companion law~\eqref{eq:eulerenergy} is satisfied in the sense of distributions on $(0,T)\times\T^d$.
\end{theorem}
Admittedly, this theorem does not bring new information in the iconic case of a polytropic pressure $p(\rho) = \kappa \rho^{\gamma_0}$, $1<\gamma_0<2$, since when the density is bounded away from zero, this gives a $C^2$ function already. However, some motivation for considering such a statement in the polytropic case was given by Akramov et.al.\ \cite{ADSW}, who allow for the formation of vacuum.  
Keeping in mind this motivating example, Akramov et.al.\ investigate sufficient conditions, which provide energy conservation for~\eqref{eq:compressEuler} when the pressure is merely continuous or at most $C^{1,\gamma-1}$. An expected (in terms of differentiability exponents) extension of the work of Feireisl et al.\ \cite{FGSW} under the slightly stronger assumption of H\"older (instead of Besov) regularity is obtained. Interestingly, this is shown regardless of the behaviour of the density near vacuum. If one insists on Besov regularity, then further assumptions on the density near vacuum are needed. The difficulty arises due to the fact that for~\eqref{eq:compressEuler} the form of the multiplier $B$ in~\eqref{eq:Compatibility} depends on the form of the pressure. Thus lowering regularity of $p$, one must deal with less regular $B$, causing additional difficulties in commutator estimates.

\footnotesize \subsection*{Acknowledgements}
This work was funded by National Science Center (Poland) grant no.\ 2017/27/B/ST1/01569. The author is grateful to Agnieszka \'Swierczewska-Gwiazda and Piotr Gwiazda for fruitful discussions and suggestions towards the final version of this text.

\end{document}